\documentclass[leqno,11pt,a4paper]{article}
\usepackage[left=2.0cm, top=2.2cm,bottom=2.5cm,right=2.0cm]{geometry}
\usepackage{amsmath,amssymb,amsthm,mathrsfs,calc,graphicx}
\usepackage[british]{babel}

\newtheorem{theorem}{Theorem}

\newtheorem{proposition}[theorem]{Proposition}

\newtheorem{remark}[theorem]{Remark}

\def\ba{\begin{array}}
\def\ea{\end{array}}
\def\bea{\begin{eqnarray} \label}
\def\eea{\end{eqnarray}}
\def\be{\begin{equation} \label}
\def\ee{\end{equation}}
\def\bit{\begin{itemize}}
\def\eit{\end{itemize}}
\def\ben{\begin{enumerate}}
\def\een{\end{enumerate}}

\def\PHT{\textup{PLT}}

\def\EE{\mathbb{E}}

\def\PP{\mathbb{P}}
\def\QQ{\mathbb{Q}}
\def\RR{\mathbb{R}}

\def\g{\gamma}
\def\d{\delta}
\def\e{\varepsilon}

\def\s{\sigma}

\def\L{\Lambda}

\def\bP{\mathbf{P}}

\def\cC{\mathcal{C}}

\def\cH{\mathcal{H}}

\def\cL{\mathcal{L}}

\def\cP{\mathscr{P}}

\def\dint{\textup{d}}

\begin{document}

\title{\bfseries Asymptotic shape of small cells}

\author{Mareen Beermann\footnotemark[1], Claudia Redenbach\footnotemark[2] and Christoph Th\"ale\footnotemark[3]}

\date{}
\renewcommand{\thefootnote}{\fnsymbol{footnote}}
\footnotetext[1]{University of Osnabr\"uck, Institute of Mathematics, Albrechstra\ss e 28a, D-49076 Osnabr\"uck, Germany. E-mail: mareen.beermann@uos.de}

\footnotetext[1]{Technical University of Kaiserslautern, Institute of Mathematics, Erwin-Schr\"odinger-Stra\ss e, D-67653 Kaiserslautern, Germany. E-mail: redenbach@mathematik.uni-kl.de}

\footnotetext[3]{
Ruhr-University Bochum, Faculty of Mathematics, NA 3/68, D-44780 Bochum, Germany. E-mail: christoph.thaele@uos.de}

\maketitle

\begin{abstract}
A stationary Poisson line tessellation is considered whose directional distribution is concentrated on two different atoms with some positive weights. The shape of the typical cell of such a tessellation is studied when its area or its perimeter tends to zero. In contrast to known results where the area or the perimeter tends to infinity, it is shown that the asymptotic shape of cells having small area is degenerate. Again in contrast to the case of large cells, the asymptotic shape of cells with small perimeter is not uniquely determined. The results are accompanied by a large scale simulation study.
\end{abstract}

\section{Introduction}

In the preface of the book \cite{SKM}, D.G.\ Kendall re-phrased a conjecture about the shape of planar tessellation cells having large area. He considered a stationary and isotropic Poisson line tessellation in the plane and conjectured that the shape of the cell containing the origin is approximately circular if its area is large. First contributions to Kendall's conjecture are due to Goldmann \cite{Goldmann}, Kovalenko \cite{Kovalenko97} and Miles \cite{MilesHeuristic}. The first result for higher dimensions is by Mecke and Osburg \cite{MeckeOsburg} who considered what they call Poisson cuboid tessellations. In a series of papers, Calka \cite{Calka2002}, Calka and Schreiber \cite{CalkaSchreiber}, Hug, Reitzner and Schneider \cite{HugReitznerSchneiderZeroPHT} and Hug and Schneider \cite{HugSchneiderDelaunay,HugSchneiderGFA} treated very general higher-dimensional versions and variants of Kendall's problem for quite general tessellation models (Poisson hyperplanes, Poisson-Voronoi and Poisson-Delaunay tessellations) and size functionals, see also the book chapter \cite{CalkaBook} for an overview. The respective results either rely on asymptotic theory for high-density Boolean models or on sharp inequalities of isoperimetric type.

\begin{figure}[t]
 \begin{center}
\includegraphics[viewport= 5.69032cm 2.83cm 11.14cm 8.25cm, clip, scale=1]{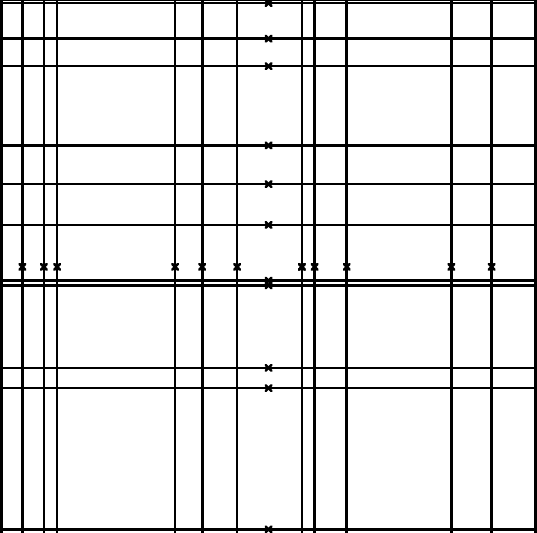}\hspace{1.3cm}
\includegraphics[viewport= 7cm 5.2cm 9.6cm 8cm, clip, angle=90.75, scale=2.1]{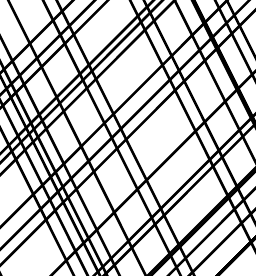}
 \caption{ \label{figRect} \small A rectangular Poisson line tessellation (left) and a non-orthogonal Poisson parallelogram tessellation (right).}
 \end{center}
\end{figure}

In this paper, we focus on the analysis of the shape of small tessellation cells. So far, we were not able to discover a general principle as the one mentioned above for the large cells behind the asymptotic geometry of small cells. For this reason, we restrict attention to the following simple model and its affine images. Take two independent stationary (homogeneous) Poisson point processes of unit intensity on the two coordinate axes in the plane and draw vertical lines through the points on the $x$-axis and horizontal lines through the points on the $y$-axis; see Figure \ref{figRect} (left). The collection of these lines (without the two coordinate axes) decomposes the plane into a countable number of non-overlapping rectangles, the collection of which is called a \textit{rectangular Poisson line tessellation}; see \cite{Favis}. Of interest in the present paper is the shape of a typical rectangle of the tessellation (the precise definition follows below). Mecke and Osburg \cite{MeckeOsburg} have shown that a typical rectangle tends to be `more and more cubical as the area tends to infinity'. In the present paper we are interested in the converse question and ask for the shape of a typical rectangle of small area. We will show that, in contrast to the large area case, the shape of typical rectangles with small area is asymptotically degenerate. Besides rectangles of small area, we also consider rectangles that have small perimeter. For such a situation we obtain a uniform distribution for our parameter measuring the shape of the rectangles (in fact not in general, but at least for the case described above). Again, this result is in contrast to the large perimeter case indicated in \cite{MeckeOsburg} (with proofs given in \cite{OsburgDiplom}). We would like to stress the fact that this is the first paper dealing with the mathematical analysis of small cells in random tessellations although some conjectures together with heuristic arguments have appeared earlier in \cite{MilesHeuristic}.

\bigskip

The rest of this text is structured as follows. In Section \ref{sec:results} we present the mathematical framework, the statements of our results, Theorem \ref{thm1} and Theorem \ref{thm:smallPerimeter}, a large scale simulation study as well as an outlook to higher space dimensions. Sections \ref{sec:ProofKantenlaengen} to \ref{sec:ProofPerimeter} are devoted to the proofs of the results.

\section{Results}\label{sec:results}

\subsection{Framework}
Denote by $\cL$ the space of lines in $\RR^2$ and by $\cL_0$ the subspace containing only lines through the origin. We let $L_1$ and $L_2$ be two different lines in $\cL_0$ and fix $q\in(0,1)$. On $\cL_0$ we define the probability measure $\QQ$ by $\QQ=q\d_{L_1}+(1-q)\d_{L_2}$, where $\d_{L_i}$ stands for the Dirac measure concentrated at $L_i$, $i=1,2$. This is to say, $\QQ$ is concentrated on $L_1$ and $L_2$ with weights $q$ and $1-q$, respectively. We also define the translation-invariant measure $\L$ on $\cL$ by the relation
\begin{equation}\label{eq:Lambda}
 \int\limits_{\cL} f(L) \L(\dint L) = \g\int\limits_{\cL_0}\int\limits_{G^\perp} f(G+x)\,\ell_{G^\perp}(\dint x)\QQ(\dint G),
\end{equation}
where $\ell_{G^\perp}$ stands for the Lebesgue measure on $G^\perp$, $\g\in(0,\infty)$ and $f:\cL\to\RR$ is a non-negative measurable function. In other words, $\L$ is concentrated on two families of lines parallel to $L_1$ and $L_2$, whereas $\gamma$ is an intensity parameter.

Let now $\eta$ be a Poisson point process on $\cL$ with intensity measure $\L$ as at \eqref{eq:Lambda}; cf.\ \cite{SW,SKM} for definitions. Clearly, the lines of $\eta$ decompose the plane into countably many parallelograms -- called \textit{cells} in the sequel -- which have pairwise no interior points in common; see Figure \ref{figRect} (right). The collection of all cells is denoted by $\cC=\cC(\eta)$ and the intersection point of the two diagonals of a parallelogram $\cP$ is denoted by $c(\cP)$. We define a probability law $\bP$ on the (measurable) space of parallelograms as follows:
\begin{equation}\label{eq:DistributionTypical}
 \bP(A) = \lim_{n\rightarrow\infty}{\EE \sum\limits_{\cP\in\cC}{\bf 1}\{\cP-c(\cP)\in A\}\,{\bf 1}\{c(\cP)\in[n]\}\over\EE\sum\limits_{\cP\in\cC}{\bf 1}\{c(\cP)\in[n]\}}\,,
\end{equation}
where $A$ is a measurable subset of parallelograms and $[n]$ stands for the centered square of area $n$. The definition \eqref{eq:DistributionTypical} formalizes the idea of a uniformly selected cell from $\cC$ (regardless of size and shape) and we call a random parallelogram with distribution $\bP$ a \textit{typical cell} of the tessellation; cf.\ \cite{SW,SKM} for background material. Note that by definition the typical cell $C$ is centered in the origin, i.e. $c(C)=0$. 

The following fact will turn out to be crucial for our further investigations. We state it here for two dimensions, its proof, however, will be presented in Section \ref{sec:ProofKantenlaengen} below for the analogous model in arbitrary space dimension. This extension is applied in Section \ref{sec:simulations}.

\begin{proposition}\label{prop:Kantenlaengen}
The edge lengths of the typical cell of a Poisson parallelogram tessellation with intensity measure $\L$ given by \eqref{eq:Lambda} are independent and exponentially distributed random variables with parameters $\g(1-q)|\cos\alpha|$ (for the edge parallel to $L_1$) and $\g q|\cos\alpha|$ (for the edge parallel to $L_2$), where $\alpha=\angle(L_1,L_2^\perp)$ is the intersection angle between $L_1$ and $L_2^\perp$.
\end{proposition}

\subsection{Results for small cells}

We consider the typical cell of a Poisson line tessellation as described above. Its random edge lengths are denoted by $X$ and $Y$ and its area by $A=XY$. To measure the shape of the typical cell we introduce two \textit{deviation functionals}. The first one is
\begin{equation}\label{eq:DefSigma}
 \sigma = 2\,{\min\{X,Y\}\over X+Y},
\end{equation}
which is a random variable taking values in $[0,1]$ (this was the reason for the choice for the factor $2$). We notice that $\sigma$ is scale invariant, i.e., $\sigma$ does not change if the parallelogram is rescaled by some constant factor. Moreover, we have $\sigma=0$ if exactly one of the edge lengths $X$ or $Y$ is zero, i.e., if the parallelogram degenerates in that it is a line segment of positive length. For single points, i.e. $X=Y=0$, $\sigma$ is not defined. As a second deviation functional we introduce
\begin{equation}\label{eq:DefTau}
 \tau = \max\{X,Y\}.
\end{equation}
This is not a scale invariant quantity, but we notice that $\tau=0$ if and only if the parallelogram is degenerated to a point.

We investigate first the asymptotic behavior of the deviation functionals $\sigma$ and $\tau$ under the condition that the typical cell area $A$ tends to zero. Our main result in this direction reads as follows.
\begin{theorem}\label{thm1}
Let $0<\e<\frac{1}{2}$. It holds that
\begin{equation}\label{ErgebnisMinRate}
\PP(\sigma>\e|A<a)=O(\sqrt{a})\qquad{\rm as}\qquad a\to 0
\end{equation}
and consequently
\begin{equation}\label{ErgebnisMinLimit}
\lim_{a\to 0}\PP(\sigma>\e|A<a)=0.
\end{equation}
Moreover, 
\begin{equation}\label{ErgebnisMax}
\lim_{a\to 0}\PP(\tau>\e|A<a)=0.
\end{equation}
\end{theorem}

Some comments are in order about the interpretation of Theorem \ref{thm1}. Firstly, \eqref{ErgebnisMinLimit} shows that the asymptotic shape of a typical cell of small area tends to that of a line segment. On the other hand, \eqref{ErgebnisMax} shows that, in the limit, this line segment cannot have positive length. 
%Together, \eqref{ErgebnisMinLimit} and \eqref{ErgebnisMax} show that the edge lengths of the typical cell decrease at different rates as $A\to 0$. In other words this means that in a first stage the typical cell shrinks basically in only one direction, which causes a typical parallelogram with small (but positive) area to look like a line segment. However, in a second stage, as the area shrinks even more, this line segment eventually degenerates to a point. 
This phenomenon is well reflected in the simulation study presented in Section \ref{sec:simulations}. We also remark that \eqref{ErgebnisMinRate} gives an upper bound for the rate of convergence.

Besides cells of small area, also cells with small perimeter can be considered. In this case the picture is somewhat different from that presented for the small area case in Theorem \ref{thm1}. In what follows we denote by $P=X+Y$ half of the perimeter length of the typical cell (the factor $1/2$ is chosen for simplicity as will become clear in the proof).
\begin{theorem}\label{thm:smallPerimeter}
Let $0<\e<1$. If $q=1/2$, $\s$ is uniformly distributed on $[0,1]$ given that $P<p$, i.e., $$\PP(\s>\e|P<p)=1-\e,$$ independently of $p$. If otherwise $q\neq 1/2$,
% \begin{equation*}
% \begin{split}
% &\PP(\s>\e,\,P<p) \\ &= {4\g_1\g_2\big((\g_1-\g_2)\e-\g_1+\g_2-(\e(\g_1-\g_2)-2\g_1)e^{-{\g_1+\g_2\over 2}p}-(\g_1+\g_2)e^{-{\e(\g_1-\g_2)-2\g_1\over 2}p}\big)\over (\g_1+\g_2)(\g_1(1-e^{-\g_2p})-\g_2(1-e^{-\g_1 p}))(\e(\g_1-\g_2)-2\g_1)}\,,
% \end{split}
% \end{equation*}
\begin{equation*}
\begin{split}
&\PP(\s>\e |\,P<p) \\ &= {4\g_1\g_2\big((\g_1-\g_2)\e +\g_1 -\g_2-(\e(\g_1-\g_2)-2\g_1)e^{-{\g_1+\g_2\over 2}p}-(\g_1+\g_2)e^{-{2\g_1- \e(\g_1-\g_2)\over 2}p}\big)\over (\g_1+\g_2)(\g_1(1-e^{-\g_2p})-\g_2(1-e^{-\g_1 p}))(\e(\g_1-\g_2)-2\g_1)}\,,
\end{split}
\end{equation*}
where $\g_1=\g(1-q)|\cos\alpha|$ and $\g_2=\g q|\cos\alpha|$ with $\alpha=\angle(L_1,L_2^\perp)$.
%Moreover, for any $p<\e$ we have that $\PP(\tau>\e|P<p)=0$ so that
%\begin{equation}\label{eq:PerimeterTau}
%\lim_{p\to 0}\PP(\tau>\e|P<p)=0.
%\end{equation}
\end{theorem}
Theorem \ref{thm:smallPerimeter} shows that the conditional deviation functional $\s$, given $P<p$, follows a uniform distribution on its range $[0,1]$ in the particular case $q=1/2$. This means that not only in contrast to the case of large perimeter (see \cite{MeckeOsburg,OsburgDiplom}), but also in contrast to the case of small area considered in Theorem \ref{thm1} above, the asymptotic shape of cells that have small perimeter is not uniquely determined. This phenomenon is well reflected by the simulation study in forthcoming Section \ref{sec:simulations}. We also refer to a related short discussion at the beginning of Section 7 in \cite{HugSchneiderGFA} about the independence of the shape of the zero cell and its perimeter. Such an interpretation becomes less obvious whenever $q\neq 1/2$. However, it is easily seen from the precise formula stated in Theorem \ref{thm:smallPerimeter} that the limit relation $$\lim_{p\to 0}\PP(\s>\e|P<p)=1-\e$$ is in order.

%Moreover, \eqref{eq:PerimeterTau} shows that in contrast to Theorem \ref{thm1} the deviation functional $\tau$ does not provide useful information about the asymptotic cell shape in the present situation since $\PP(\tau>\e|P<p)=0$ for any $p<\e$.

\begin{figure}[t]
\begin{center}
\includegraphics[width=0.45\textwidth]{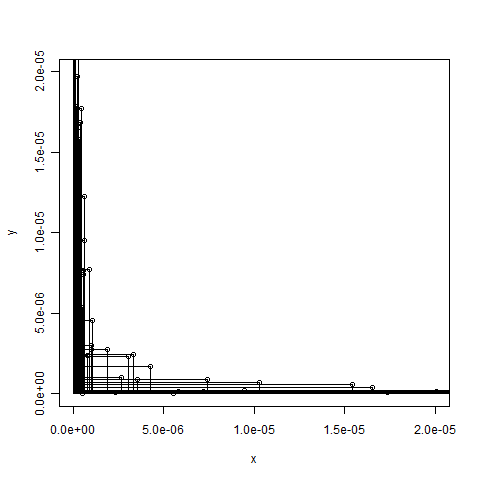}
\includegraphics[width=0.45\textwidth]{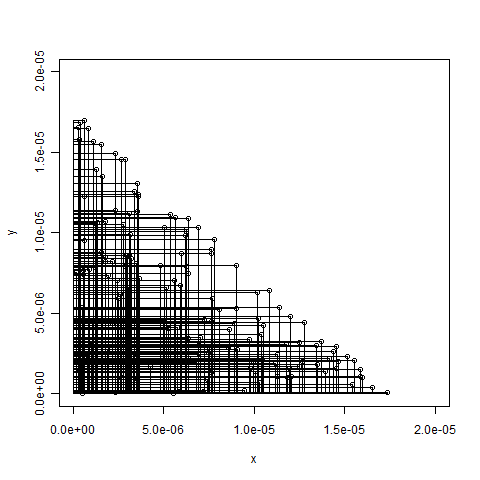}
\caption{\label{fig2D} \small The $150$ cells with smallest area (left) and smallest perimeter (right). Note that the axes on the left are cut off. The maximal edge length observed is $0.09$.}
\end{center}
\end{figure}

\subsection{Simulation results and outlook to higher space dimensions}\label{sec:simulations}

\begin{figure}[t]
\begin{center}

\includegraphics[width=0.45\textwidth]{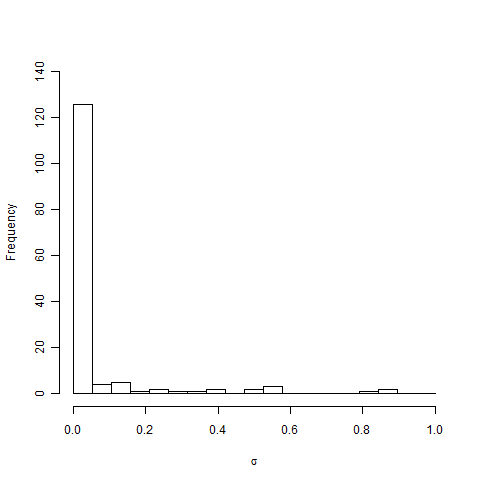}
\includegraphics[width=0.45\textwidth]{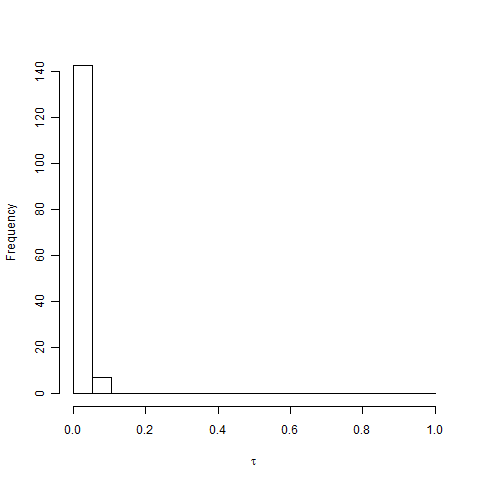}
\caption{\label{figArea2D} \small Histograms for the deviation functionals of the $150$ cells with smallest area.}
\end{center}
\end{figure}

To highlight and to underpin the theoretical results in the previous subsection we performed the following simulation study. We chose $\g=2$, $q=1/2$ and $\QQ={1\over 2}\big(\d_{L_x}+\d_{L_y}\big)$, where $L_x$ and $L_y$ are the two orthogonal coordinate axes. In this case, the edge lengths of the typical cell of the rectangular Poisson line tessellation are independently exponentially distributed with mean $1$; see Proposition \ref{prop:Kantenlaengen}. Hence, we simulated $10^{12}$ independent realizations of the random vector $(X,Y)$ with $X$ and $Y$ i.i.d.\ Exponential(1). From the collection of cells obtained this way, we selected the $150$ cells with smallest area. These are shown in Figure \ref{fig2D} (left). Histograms of the deviation functionals $\sigma$ and $\tau$ are shown in Figure \ref{figArea2D}. Both the line segment shape of the cells with an accumulation around the origin and the peak at zero in the histograms for the deviation functionals are nicely visible.

As discussed above, the area is only one measure of size of a cell. We investigate now the shape of small cells in the sense that their perimeter tends to zero. Therefore, the $150$ cells with smallest perimeter were extracted from the sample generated above. These cells together with their deviation functionals are shown in Figure \ref{fig2D} (right) and Figure \ref{figPerim2D}, respectively. In this situation, the maximum of the edge lengths also tends to zero, but Theorem \ref{thm:smallPerimeter} implies that $\sigma$ follows a uniform distribution on the interval $[0,1]$, which is also nicely visible in the histograms. The minimal and maximal sizes  of area and perimeter of the cells included in the statistics described above are given in Table \ref{tabMinMax}. 

\begin{table}[h]
\begin{center}
\begin{tabular}{|l||r|r|}
\hline
&minimum & maximum\\
\hline
\hline
area 2D& 1.79e-14&8.46e-12\\
perimeter 2D&1.06e-6&3.52e-5\\
\hline
\end{tabular}
\caption{\label{tabMinMax} \small Minimum and maximum of the size of the cells considered in the statistics (planar case).}
\end{center}
\end{table}

\begin{figure}[t]
\begin{center}
\includegraphics[width=0.45\textwidth]{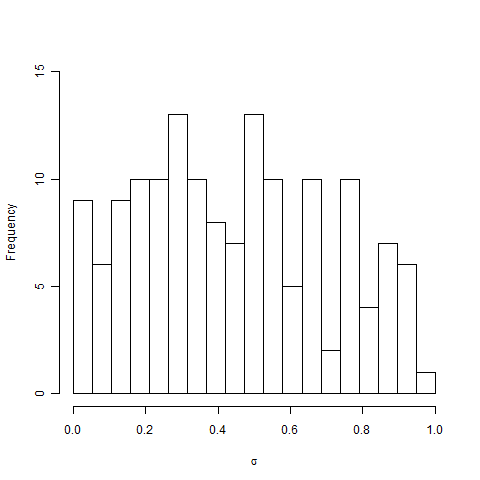}
\includegraphics[width=0.45\textwidth]{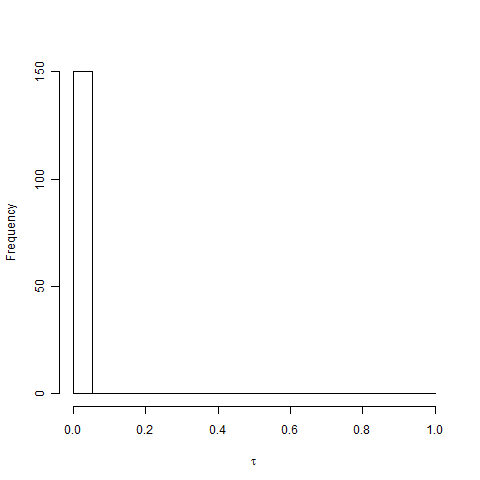}
\caption{\label{figPerim2D} \small Histograms for the deviation functionals of the 150 cells with smallest perimeter.}
\end{center}
\end{figure}

The Poisson line tessellations considered in this paper have natural analogues in higher space dimensions, the Poisson cuboid tessellations for which we refer to \cite{Favis} and to Section \ref{sec:ProofKantenlaengen} below. Also for these tessellations the question about the shape of small cells can be asked. Natural candidates are the typical cell of small volume, small surface area or small total edge length. Unfortunately, we were not able to derive the higher-dimensional pendants to Theorem \ref{thm1} or Theorem \ref{thm:smallPerimeter} in full generality. This is mainly due to technical complications that arise for space dimensions $d\geq 3$ and cause that the Abelian-type theorem used in the proof of Theorem \ref{thm1} can no more be applied. For this reason we carried out a simulation study concerning small cells in $\RR^3$. For this purpose, $10^{12}$ independent realizations of the random vector $(X_1,X_2,X_3)$ with $X_1,X_2,X_3$ i.i.d.\ Exponential(1) were generated representing the random edge lengths of the typical cell (this is justified by Proposition \ref{prop:KantenlaengenHoeherdimensional} below). From the sample of cells the $150$ cells with smallest volume, surface area and total edge length were extracted. Histograms for the deviation functionals $\sigma=3\min(X_1,X_2,X_3)/(X_1+X_2+X_3)$ (again the factor 3 yields a value of 1 for a cube and the range $[0,1]$) and $\tau=\max(X_1,X_2,X_3)$ for these cells are shown in Figure \ref{figVolume3D}. The minimal and maximal sizes of the cells included in the statistics above are summarized in Table \ref{tabMinMax2}. 

\begin{table}[h]
\begin{center}
\begin{tabular}{|l||r|r|}
\hline
&minimum & maximum\\
\hline
\hline
volume 3D&3.20 e-15&4.97e-13\\
surface area 3D&1.74e-8&3.58e-7\\
total edge length 3D&6.80e-4&3.88e-3\\
\hline
\end{tabular}
\caption{\label{tabMinMax2} \small Minimum and maximum of the size of the cells considered in the statistics (spatial case).}
\end{center}
\end{table} 

\begin{figure}[t]
\begin{center}
\includegraphics[width=0.3\textwidth]{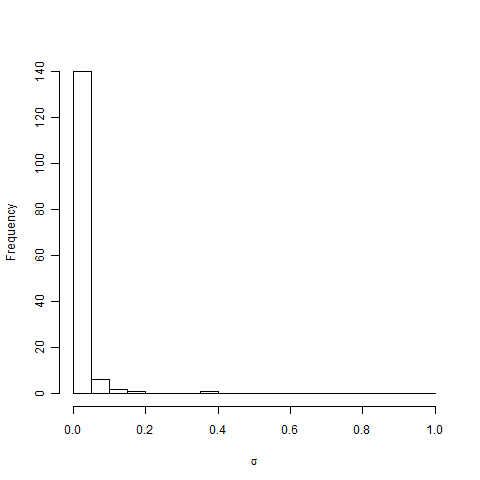}
\includegraphics[width=0.3\textwidth]{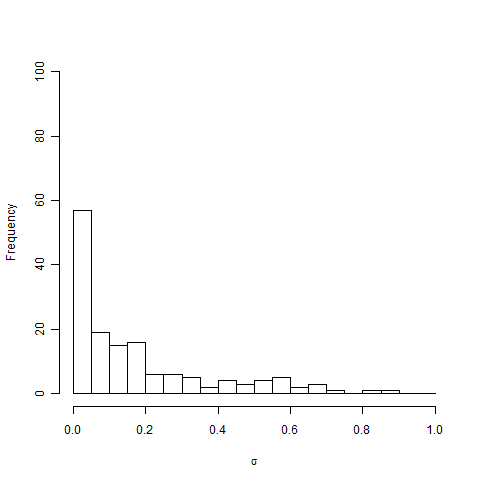}
\includegraphics[width=0.3\textwidth]{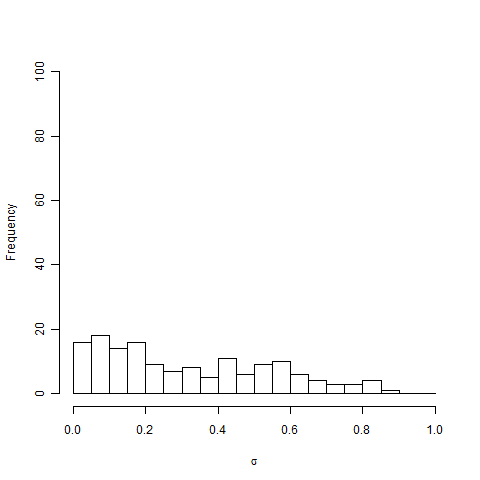}\\
\includegraphics[width=0.3\textwidth]{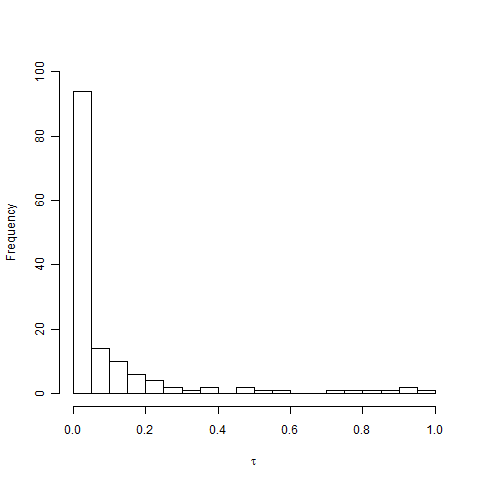}
\includegraphics[width=0.3\textwidth]{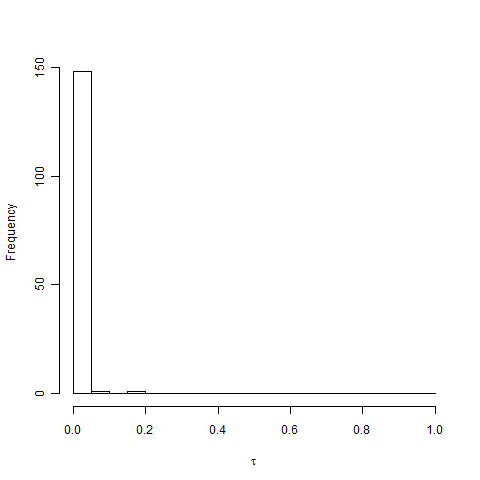}
\includegraphics[width=0.3\textwidth]{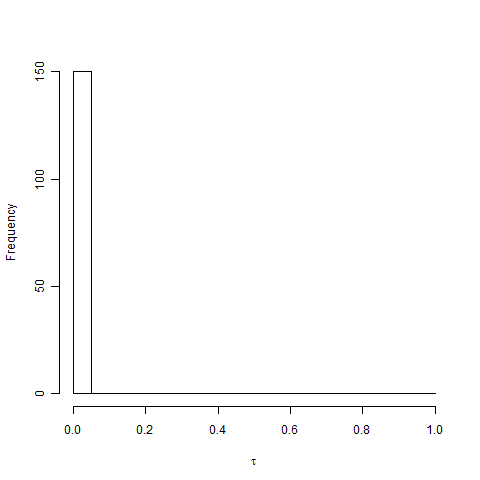}
\caption{\label{figVolume3D} \small Histograms for the deviation functionals $\sigma$ (first row) and $\tau$ (second row) of the 150 cells with smallest volume (left), surface area (middle) and total edge length (right) in $\RR^3$.}
\end{center}
\end{figure}

\section{Proof of Proposition \ref{prop:Kantenlaengen} and its higher-dimensional extension}\label{sec:ProofKantenlaengen}

As announced in Section \ref{sec:results}, we will formulate and prove here a higher-dimensional version of Proposition \ref{prop:Kantenlaengen}. To state the result, we first need to introduce the higher-dimensional model. So, fix a space dimension $d\geq 2$, let $u_1,\ldots,u_d$ be linearly independent unit vectors in $\RR^d$ and fix weights $q_1,\ldots,q_d\in[0,1]$ such that $q_1+\ldots+q_d=1$. We define linear hyperplanes $H_1,\ldots,H_d$ by $$H_i={\rm span}(\{u_1,\ldots,u_d\}\setminus\{u_{d-i+1}\}),\qquad i\in\{1,\ldots,d\}\,,$$ and the probability measure $\QQ$ on the space $\cH_0$ of hyperplanes through the origin by putting $\QQ=q_1\d_{H_1}+\ldots+q_d\d_{H_d}$. The translation-invariant measure $\L$ on the space $\cH$ of (affine) hyperplanes in $\RR^d$ induced by $\QQ$ is given by $$\int\limits_{\cH}f(H)\,\L(\dint H)=\g\int\limits_{\cH_0}\int\limits_{H_0^\perp}f(H_0+x)\,\ell_{H_0^\perp}(\dint x)\QQ(\dint H_0)\,,$$ where $0<\g<\infty$ is a fixed constant and where $f:\cH\to\RR$ is non-negative and measurable. Note that taking $d=2$ and $q_1=1-q_2=q$ we get back the set-up described in Section \ref{sec:results} for the planar case $d=2$.

Now, let $\eta$ be a Poisson point process on $\cH$ with intensity measure $\L$ as defined above. By abuse of notation, we will identify $\eta$ with the random closed set in $\RR^d$, which is induced by the union of all hyperplanes in $\eta$. They decompose $\RR^d$ into a countable set of random parallelepipeds and the distribution of the typical cell (parallelepiped) of this tessellation is defined similarly as in \eqref{eq:DistributionTypical}. For $i\in\{1,\ldots,d\}$ let $L_i={\rm lin}(u_i)$ be the line spanned by $u_i$. Then the discussion around \cite[Equation (6.3)]{MilesFlats} together with \cite[Theorem 4.4.7]{SW} shows that $\eta\cap L_i$ is a homogeneous Poisson point process on $L_i$ of intensity $$\g_{L_i}=\g\sum_{j=1}^d q_j\,|\cos\angle(L_i,H_j^\perp)|\,,$$ where $\angle(L_i,H_j^\perp)$ is the angle between $L_i$ and $H_j^\perp$ ($i,j\in\{1,\ldots,d\}$). In the particular planar case $d=2$ we have $\g_{L_1}=\g(1-q)|\cos\angle(L_1,L_2^\perp)|=\g(1-q)|\cos\alpha|$ and $\g_{L_2}=\g q|\cos\angle(L_2,L_1^\perp)|=\g q|\cos\alpha|$ with $\alpha=\angle(L_1,L_2^\perp)$. We can now state the higher-dimensional version of Proposition \ref{prop:Kantenlaengen}.

\begin{proposition}\label{prop:KantenlaengenHoeherdimensional}
The edge lengths of a typical cell of a Poisson cuboid tessellation induced by $\eta$ are independent and exponentially distributed random variables with parameters $\g_{L_1},\ldots,\g_{L_d}$, respectively.
\end{proposition}
\begin{proof}
As a first step let us describe an alternative construction for the random set $\eta$, which in the planar case has already been considered in the introduction. Recall the definition of the lines $L_i$ from above and let for each $i\in\{1,\ldots,d\}$, $\xi_i$ be a homogeneous Poisson point process on $L_i$ with intensity $\g_{L_i}$. We assume that $\xi_1,\ldots,\xi_d$ are independent. Now, for each $i\in\{1,\ldots,d\}$, place hyperplanes through the Poisson points on $L_i$ orthogonal to $L_i$. The collection (or union) of all hyperplanes constructed this way has the same distribution as $\eta$.

As a next step we describe a construction of the typical cell. To carry this out, we denote by $Z_0$ the almost surely uniquely determined $d$-dimensional parallelepiped of the tessellation induced by $\eta$ that contains the origin. Then $Z_0$ is divided by the hyperplanes $H_1,\ldots,H_d$ into $2^d$ smaller parallelepipeds meeting at the origin. With probability one, exactly one of these parallelepipeds, $Z$ say, has the property that for all its corners the first coordinate is non-negative. Now, Theorem 10.4.7 in \cite{SW} implies that (up to translations) $Z$ has distribution $\bP$ defined by (the higher-dimensional analogue of) \eqref{eq:DistributionTypical}. In other words, $Z$ has (again up to translations) the same distribution as the typical cell of the Poisson hyperplane tessellation induced by $\eta$, see also Section 4 in \cite{Mecke99}. Note in particular that the random parallelepiped $Z$ has one of its corners at the origin. In view of the construction of $\eta$ described at the beginning of the proof, this implies that the edge lengths of $Z$ are the distances from the origin of $d$ independent and homogeneous Poisson point processes on $L_1,\ldots,L_d$ with intensities $\g_{L_1},\ldots,\g_{L_d}$, respectively, to their next point on the left or right (depending on the position of $Z$ within $Z_0$). Thus, standard properties of such point processes allow us to conclude that the edge-lengths of the typical parallelepiped are independent and exponentially distributed with parameters $\g_{L_1},\ldots,\g_{L_d}$.
\end{proof}

\section{Proof of Theorem \ref{thm1}}\label{sec:proofThm1}

\subsection{Reduction}

We claim that without loss of generality we can restrict the proof of Theorem \ref{thm1} to the case $\g=2$, $q=1/2$ and $\QQ={1\over 2}\big(\d_{L_x}+\d_{L_y}\big)$, where $L_x$ and $L_y$ are the two orthogonal coordinate axes. To show this, let us denote such a tessellation by $\PHT^*$ and a tessellation with general parameters $\g$, $q$ and $\QQ$ by $\PHT(\g,q,\QQ)$. We notice now that due to our assumptions on $\g$ and $q$ there exists a non-degenerate linear transformation $f=f(\g,q,\QQ):\RR^2\to\RR^2$ such that $\PHT(\g,q,\QQ)$ after application of $f$ has the same distribution as $\PHT^*$. We notice further that the images under $f$ of a parallelogram, a line segment and a point are again a parallelogram, a line segment and a point, respectively. Thus, the statement of Theorem \ref{thm1} is invariant under non-degenerate linear transformations of the underlying Poisson line tessellation. This implies that it is sufficient to establish the statement for one special choice of $\g$, $q$ and $\QQ$, viz.\ $\PHT^*$.

\subsection{Proof for $\PHT^*$}

%For the tessellation $\PHT^*$ the random edge lengths $X$ and $Y$ of the typical cell -- which is a rectangle in this case -- are independent and identically distributed according to an exponential distribution with mean $1$; see \cite{Favis}. Actually, this was the reason for the choice $\g=2$ and $p=1/2$ above.

\paragraph{Proof of \eqref{ErgebnisMinRate}.}

To simplify the calculations we work from now on with $\widehat\s=\s/2$ and translate the result afterwards to the original deviation functional $\s$. To start with the calculation, we write the conditional probability as $${\Bbb P}(\widehat\sigma>\e|A<a)=\frac{{\Bbb P}(\widehat\sigma >\e,A<a)}{{\Bbb P}(A<a)}.$$ In what follows, we consider the numerator and the denominator  separately. To deal with the numerator we have to consider the event $\hat{\sigma}={\min\{X,Y\}\over X+Y} > \e$ and $A=XY<a$. Without loss of generality we can assume that $\min\{X,Y\}=Y$. The condition ${\min\{X,Y\}\over X+Y} > \e$ then becomes ${Y\over X+Y} > \e$ and implies $\frac{\e}{1-\e}X<Y<X$. Therefore $\e < {1\over 2}$ has to apply to retain $Y$ as the minimum. 
%Since ${\min\{X,Y\}\over X+Y} > \e$ has to be fulfilled, also ${X\over X+Y}>\e$ must be in order, meaning that $Y<\frac{(1-\e)}{\e}X$, since $Y$ is determined as the minimum $Y<X$ is the more restrictive condition. 
Taking the second condition $A=XY<a$ into account, leads to ${\displaystyle \min\left\{X, \frac{a}{X}\right\}}$ as the upper bound for $Y$. Considering the lower bound for $Y$, that is $Y=\frac{\e}{1-\e}X$, we get the upper bound for $X$ from the condition $XY<a$. Thus, the numerator can be written as
\begin{eqnarray*}
\iint \limits_{{\widehat\sigma >\e \atop A< a}}e^{-x-y}\,\dint y\dint x &=& 2 \int\limits_{0}^{\sqrt{\frac{a(1-\e)}{\e}}}\int\limits_{\frac{x\e}{1-\e}}^{\min\left\{x, \frac{a}{x}\right\}}e^{-x-y}\,\dint y\dint x\\
&=&2 \int\limits_{0}^{\sqrt{a}}\int\limits_{\frac{\e x}{1-\e}}^{x}e^{-x-y}\,\dint y\dint x+2 \int\limits_{\sqrt{a}}^{\sqrt{\frac{a(1-\e)}{\e}}}\int\limits_{\frac{\e x}{1-\e}}^{\frac{a}{x}}e^{-x-y}\,\dint y\dint x\\
&\leq& 2 \int\limits_{0}^{\sqrt{a}}
\int\limits_{\frac{\e x}{1-\e}}^{x}\,\dint y\dint x+2 \int\limits^{\sqrt{\frac{a(1-\e)}{\e}}}_{\sqrt{a}}\int\limits_{\frac{\e x}{1-\e}}^{\frac{a}{x}}\,\dint y\dint x\,.
\end{eqnarray*}
%\begin{eqnarray*}
%\iint \limits_{{\widehat\sigma >\e \atop A< a}}e^{-x-y}\,\dint y\dint x &=& \int_{0}^{\sqrt{\frac{a(1-\e)}{\e}}}\int\limits_{\frac{x\e}{1-\e}}^{\min\left\{\frac{x(1-\e)}{\e}, \frac{a}{x}\right\}}e^{-x-y}\,\dint y\dint x\\
%&=&\int\limits_{0}^{\sqrt{\frac{a\e}{1-\e}}}\int\limits_{\frac{\e x}{1-\e}}^{\frac{x(1-\e)}{\e}}e^{-x-y}\,\dint y\dint x+\int\limits_{\sqrt{\frac{a\e}{1-\e}}}^{\sqrt{\frac{a(1-\e)}{\e}}}\int\limits_{\frac{\e x}{1-\e}}^{\frac{a}{x}}e^{-x-y}\,\dint y\dint x\\
%&\leq& \int\limits_{0}^{\sqrt{\frac{a\e}{1-\e}}}
%\int\limits_{\frac{\e x}{1-\e}}^{\frac{x(1-\e)}{\e}}\,\dint y\dint x+\int\limits^{\sqrt{\frac{a(1-\e)}{\e}}}_{\sqrt{\frac{a\e}{1-\e}}}\int\limits_{\frac{\e x}{1-\e}}^{\frac{a}{x}}\,\dint y\dint x\,.
%\end{eqnarray*}
The last expression can be determined by a straight forward integration procedure, which yields that asymptotically, as $a\to 0$, it behaves like $a$ times a constant depending on $\e$ (the two logarithmic terms cancel out).

Using the substitution $u=x\frac{1}{\sqrt{a}}$ and $v=y\frac{1}{\sqrt{a}}$ we can write the denominator as
\begin{eqnarray*}
\iint\limits_{A<a} e^{-x-y}\,\dint y\dint x &=&\int\limits_{0}^{\infty}\int\limits_{0}^{\frac{a}{x}}e^{-(x+y)}\,\dint y\dint x= \int\limits_{0}^{\infty}\int\limits_{0}^{\frac{1}{u}}e^{-\sqrt{a}(u+v)}a\,\dint v\dint u\,.
\end{eqnarray*}
Applying the substitution $s=u+v$ and $t=u-v$ in the next step, we get
\begin{eqnarray}\label{eq:ZeischenschrittXXXYYY}
\iint\limits_{A<a} e^{-x-y}\,\dint y\dint x&=&\frac{a}{2}\int\limits_{0}^{\infty}\int\limits_{-s}^{s}e^{-\sqrt{a}s}\,{\bf 1}(s^2-t^2<4)\,\dint t\dint s\,.
\end{eqnarray}
Now we split this double integral and calculate the resulting integrals directly as far as possible:
\begin{eqnarray}
\nonumber &&\frac{a}{2}\int\limits_{0}^{\infty}\int\limits_{-s}^{s} e^{-\sqrt{a}s}\,{\bf 1}(s^2-t^2<4)\,\dint t\dint s\\
\nonumber &=&\frac{a}{2}\int\limits_{0}^{\infty}e^{-\sqrt{a}s}\left(\;\int\limits_{-s}^{s}{\bf 1}(s\leq 2)\,\dint t+\int\limits_{-s}^{s}{\bf 1}(   |t|\geq\sqrt{s^{2}-4}){\bf 1}(s> 2)\,\dint t\right)\dint s\\
\nonumber &=&a\left(\int\limits_{0}^{2}se^{-\sqrt{a}s}\,\dint s+\int\limits_{2}^{\infty}e^{-\sqrt{a}s
}(s-\sqrt{s^{2}-4})\,\dint s\right)\\
&=&1-a\int\limits_{2}^{\infty}e^{-\sqrt{a}s
}\,\sqrt{s^{2}-4}\,\dint s.\label{eq:proofA1}
\end{eqnarray}
Even if the integral in \eqref{eq:proofA1} looks rather innocent, its asymptotic behavior as $a\to 0$ turns out to be not accessible with elementary methods as above. To overcome this difficulty we make use of a theorem of Abelian type. So, let $F(s):=\sqrt{s^{2}+4s}$ and write $${\cal I}(a):=a\int\limits_{2}^{\infty}e^{-\sqrt{a}s}\,\sqrt{s^{2}-4}\,\mathrm{d}s
=a e^{-2\sqrt{a}}\int\limits_{0}^{\infty}e^{-\sqrt{a}s}\,F(s)\,\dint s\,,$$which arises by a shift $s\mapsto s+2$,
and use now a relation between the Laplace transformation and the Laplace-Stieltjes transformation \cite[Theorem 2.3a]{Widder} to conclude from the Theorem of Abelian type \cite[Theorem 8.5.2]{Kawata} (or more precisely from Corollary 8.5.1 ibidem) that
\begin{equation}\label{eq:proofA2}
\lim_{a\to 0}{{\cal I}(a)\over e^{-2\sqrt{a}}}=1.
\end{equation}
Putting together \eqref{eq:proofA1} and \eqref{eq:proofA2} we see with a Taylor expansion of the exponential function that ${\Bbb P}(A<a)$ behaves asymptotically like $\sqrt{a}$ in the limit as $a\to 0$. Combining this with the asymptotic behavior of the numerator implies \eqref{ErgebnisMinRate} for $\widehat\s$ as well as for the original deviation functional $\s$.\hfill $\Box$

%\begin{remark}\label{rem:Optimality}\rm
%At the beginning of the proof we used the very rough estimate $e^{-x-y}\leq 1$ for the integrand in the numerator. This can be %refined considerably by explicit calculations. This is possible except of one integral whose asymptotic behavior can unfortunately %not be discussed with an Abelian theorem. Moreover, all estimates we found did not improve the rate $O(\sqrt{a})$ as $a\to 0$. %Progress in this direction would be necessary to optimize \eqref{ErgebnisMinRate}.
%\end{remark}

\paragraph{Proof of \eqref{ErgebnisMax}.}

We start by re-writing the conditional probability as $${\Bbb P}(\tau>\e|A< a)=\frac{{\Bbb P}(\max\{X,Y\}>\e,A<a)}{{\Bbb P}(A< a)}.$$ The denominator is the same as in the proof of \eqref{ErgebnisMinRate} and behaves like $\sqrt{a}$ as $a\to 0$. Next, we consider the numerator and write
\begin{equation*}
\begin{split}
{\Bbb P}(\max\{X,Y\}>\e,A<a) &\leq {\Bbb P}(X>\e, A< a)+{\Bbb P}(Y>\e, A< a)\\
&= 2{\Bbb P}(X>\e, A< a)=2\int_{\e}^{\infty}\int_{0}^{\frac{a}{x}}e^{-x-y}\,\dint y\dint x.
\end{split}
\end{equation*}
Since numerator and denominator both tend to $0$ as $a\to 0$, we may apply l'Hospital's rule. For the numerator we find
\begin{align*}
\frac{\dint}{\dint a}\int_{\e}^{\infty}\int_{0}^{\frac{a}{x}}e^{-x-y}\,\dint y\dint x &=\int_{\e}^{\infty}\frac{e^{-\frac{a}{x}-x}}{x}\,\dint x \to \Gamma(0,\e)<\infty\quad{\rm as}\quad a\to 0,
\end{align*}
where $\Gamma(\,\cdot\,,\,\cdot\,)$ stands for the lower incomplete $\Gamma$-function.

Next, we turn to the denominator. Using \eqref{eq:ZeischenschrittXXXYYY} and \eqref{eq:proofA1}, differentiation yields
\begin{equation*}
\begin{split}
\frac{\dint}{\dint a}\int_{0}^{\infty}\int_{0}^{\frac{a}{x}}e^{-x-y}\,\dint y\dint x &=\frac{\dint}{\dint a}\left(1-a\int\limits_{2}^{\infty}e^{-\sqrt{a}s
}\,\sqrt{s^{2}-4}\,\dint s\right)\\
&=\frac{1}{2}\int\limits_{2}^{\infty}e^{-\sqrt{a}s}\,(-2+\sqrt{a}\,s)\,\sqrt{s^2-4}\;\dint s.
\end{split}
\end{equation*}
Now we split this integral into two parts and apply the shift $s\mapsto s+2$. This leads to
\begin{equation*}
\begin{split}
\frac{\dint}{\dint a}\int_{0}^{\infty}\int_{0}^{\frac{a}{x}}e^{-x-y}\,\dint y\dint x &= -\int_{2}^{\infty}e^{-\sqrt{a}s}\,\sqrt{s^2-4}\,\dint s +\frac{1}{2}\int_{2}^{\infty}e^{-\sqrt{a}s}\,\sqrt{a}\,s\,\sqrt{s^2-4}\,\dint s\\
&=-e^{-2\sqrt{a}}\int_{0}^{\infty}e^{-\sqrt{a}s}\sqrt{s^2+4s}\,\dint s\\
&\hspace{1.7cm} +\frac{\sqrt{a}}{2}e^{-2\sqrt{a}}\int_{0}^{\infty}e^{-\sqrt{a}s}\,(s+2)\,\sqrt{s^2+4s}\,\dint s.
\end{split}
\end{equation*}
Rearrangement then results in
\begin{equation*}
\begin{split}
\frac{\dint}{\dint a}\int_{0}^{\infty}\int_{0}^{\frac{a}{x}}e^{-x-y}\,\dint y\dint x &=- e^{-2\sqrt{a}}\int_{0}^{\infty}e^{-\sqrt{a}s}\,\sqrt{s^2+4s}\,\dint s\\
&\hspace*{1.7cm}+\sqrt{a}e^{-2\sqrt{a}}\int_{0}^{\infty}e^{-\sqrt{a}s}\,\sqrt{s^2+4s}\,\dint s\\
&\hspace*{1.7cm} +\frac{\sqrt{a}}{2}e^{-2\sqrt{a}}\int_{0}^{\infty}e^{-\sqrt{a}s}\,s\,\sqrt{s^2+4s}\,\dint s\\
&=\underbrace{(\sqrt{a}-1)e^{-2\sqrt{a}}\int_{0}^{\infty}e^{-\sqrt{a}s}\,\sqrt{s^2+4s}\,\dint s}_{=:T_1}\\ &\hspace{1.7cm}+\underbrace{\frac{\sqrt{a}}{2}e^{-2\sqrt{a}}\int_{0}^{\infty}e^{-\sqrt{a}s}\,s\,\sqrt{s^2+4s}\,\dint s}_{=:T_2}.
\end{split}
\end{equation*}
The term $T_1$ has already been discussed in the proof of \eqref{ErgebnisMinRate} above. Thus, we may concentrate on $T_2$ and write $T_2(a)$ to indicate its dependence on $a$. Defining $F(s):=s\sqrt{s^2+4s}$ we have that
\begin{equation*}
T_2(a)=\frac{\sqrt{a}}{2}e^{-2\sqrt{a}}\int_{0}^{\infty}e^{-\sqrt{a}s}F(s)\,\mathrm{d}s
\end{equation*}
and can thus apply \cite[Corollary 8.5.1]{Kawata} again, which gives that $T_2(a)$ behaves asymptotically like $e^{-2\sqrt{a}}/a$, as $a\to 0$. Putting things together we see that the numerator, after differentiation, is bounded, whereas the denominator, again after differentiation, tends to $+\infty$ as $a\to 0$, implying that $$\lim_{a\to 0}{\Bbb P}(\tau>\e|A\leq a)=0.$$ This completes the proof of \eqref{ErgebnisMax}.\hfill $\Box$

\section{Proof of Theorem \ref{thm:smallPerimeter}}\label{sec:ProofPerimeter}

Applying the reduction step as in the proof of Theorem \ref{thm1} is not possible here since the problem involving a small perimeter is not invariant under non-degenerate linear transformation. However, we can use Proposition \ref{prop:Kantenlaengen} to give a direct proof. It implies that the random edge lengths $X$ and $Y$ of the typical cell (parallelogram) are independent and identically distributed according to an exponential distribution with mean $\g_1:=\g_{L_1}$ and $\g_2:=\g_{L_2}$. If $\g_1=\g_2=\g$ it follows that the half perimeter length $P=X+Y$ is Erlang distributed with parameters $\g$ and $2$. In fact this was the reason for the factor $1/2$ in the definition of $P$. Thus,
\begin{equation}\label{eq:VerteilungP}
\PP(P<p)=1-(1+\g p)e^{-\g p}\qquad{\rm for\ any}\qquad p>0.
\end{equation}
If (without loss of generality) $\g_1>\g_2$, $P=X+Y$ has distribution function

\begin{equation}\label{eq:VerteilungPUngleich}
\PP(P<p)=1-{\g_1e^{-\g_2 p}-\g_2e^{-\g_1p}\over\g_1-\g_2}\qquad{\rm for\ any}\qquad p>0.
\end{equation}

Without loss of generality assume that $X>Y$ and have a closer look at the event $\{\s>\e,\,P<p\}$. If $X\in[0,p/2]$, then $Y$ may range between ${\e\over 2-\e}X$ and $X$, and if $X\in[p/2,p(1-\e/2)]$ then $Y$ ranges between ${\e\over 2-\e}X$ and $p-X$ (the remaining case $X>p(1-\e/2)$ contradicts $X+Y<p$ and $Y>{\e\over 2-\e}X$). Thus, $$\PP(\s>\e,\,P<p)=2\int\limits_0^{p/2}\int\limits_{{\e\over 2-\e}x}^{x}\g_1\g_2e^{-\g_1x-\g_2y}\,\dint y\dint x+2\int\limits_{p/2}^{p(1-\e/2)}\int\limits_{{\e\over 2-\e}x}^{p-x}\g_1\g_2e^{-\g_1x-\g_2y}\,\dint y\dint x.$$
If $\g_1=\g_2=\g$, evaluation of these integrals yields $$\PP(\s>\e,\,P<p)=(1-\e)(1-(1+\g p)e^{-\g p})$$ and in view of \eqref{eq:VerteilungP} the exact distributional result $\PP(\s>\e|P<p)=1-\e.$ If otherwise $\g_1>\g_2$, one shows that
\begin{equation*}
\begin{split}
&\PP(\s>\e |\,P<p) \\ &= {4\g_1\g_2\big((\g_1-\g_2)\e +\g_1 -\g_2-(\e(\g_1-\g_2)-2\g_1)e^{-{\g_1+\g_2\over 2}p}-(\g_1+\g_2)e^{-{2\g_1- \e(\g_1-\g_2)\over 2}p}\big)\over (\g_1+\g_2)(\g_1(1-e^{-\g_2p})-\g_2(1-e^{-\g_1 p}))(\e(\g_1-\g_2)-2\g_1)}\,,
\end{split}
\end{equation*}

which reduces to the separately treated uniform distribution if $\g_1=\g_2$. It remains to notice that in view of Proposition \ref{prop:Kantenlaengen}, $\g_1=\g_2$ if and only if $q=1/2$.
%Next, we notice that whenever $p<\e$ we have the following situation (without loss of generality we may assume that $\max(X,Y)=X$): we want that $X>\e$ and $X+Y<p<\e$ at the same time, which is impossible. Consequently, $\PP(\tau>\e,\,P<p)=0$ holds for such $p$.
This completes the proof.\hfill $\Box$

\begin{remark}\rm
We would like to point out that the first result of Theorem \ref{thm:smallPerimeter} has a well-known background. Namely, let $X$ and $Y$ be two independent and exponentially distributed random variables. Then $X$ (or $Y$), given that $X+Y=s$ for some fixed $s>0$, is uniformly distributed on $[0,s]$.
\end{remark}

\subsection*{Acknowledgement}
We would like to thank two anonymous referees for their hints and suggestions. They were very helpful for us to improve the text.

\end{document}